\documentclass[a4paper]{article}

\usepackage[utf8]{inputenc}
\usepackage{lmodern}
\usepackage{geometry}
\usepackage{array}
\usepackage{graphicx}
\usepackage{amssymb, amsfonts, amsthm, amsmath}
\usepackage[english]{babel}
\usepackage[pdfborder={0 0 0}]{hyperref}
\usepackage{tikz}
\usetikzlibrary{decorations}
\usetikzlibrary{decorations.pathreplacing}
\tikzset{individu/.style={draw,thick}}

\usepackage{subfigure}

\theoremstyle{plain}
\newtheorem{theorem}{Theorem}[section]
\newtheorem{corollary}[theorem]{Corollary}
\newtheorem{lemma}[theorem]{Lemma}
\newtheorem{proposition}[theorem]{Proposition}

\theoremstyle{definition}

\theoremstyle{remark}
\newtheorem{remark}[theorem]{Remark}

\newcommand{\N}{\mathbb{N}}

\newcommand{\R}{\mathbb{R}}

\newcommand{\ind}[1]{\mathbf{1}_{\left\{#1\right\}}}

\newcommand{\floor}[1]{{\left\lfloor #1 \right\rfloor}}
\newcommand{\ceil}[1]{{\left\lceil #1 \right\rceil}}

\newcommand{\calC}{\mathcal{C}}

\numberwithin{equation}{section}

\DeclareMathOperator{\E}{\mathbb{E}}

\renewcommand{\P}{\mathbb{P}}

\newcommand{\calF}{\mathcal{F}}

\newcommand{\calL}{\mathcal{L}}

\newcommand{\calG}{\mathcal{G}}

\newcommand{\calB}{\mathcal{B}}

\renewcommand{\bar}[1]{\overline{#1}}

\newcommand{\T}{\mathbf{T}}
\renewcommand{\tilde}[1]{\widetilde{#1}}
\renewcommand{\hat}[1]{\widehat{#1}}

\renewcommand{\rho}{\varrho}
\renewcommand{\epsilon}{\varepsilon}

\title{$N$-Branching random walk with $\alpha$-stable spine}
\author{Bastien Mallein\footnote{LPMA, UPMC (Paris 6) and DMA, ENS.  Research partially supported by the ANR project MEMEMO2.}}
\date{\today}
\newcommand{\calM}{\mathcal{M}}

\begin{document}

\maketitle

\begin{abstract}
We consider a branching-selection particle system on the real line, introduced by Brunet and Derrida in \cite{BrD97}. In this model the size of the population is fixed to a constant $N$. At each step individuals in the population reproduce independently, making children around their current position. Only the $N$ rightmost children survive to reproduce at the next step. Bérard and Gouéré studied the speed at which the cloud of individuals drifts in \cite{BeG10}, assuming the tails of the displacement decays at exponential rate; Bérard and Maillard \cite{BeM} took interest in the case of heavy tail displacements. We take interest in an intermediate model, considering branching random walks in which the critical spine behaves as an $\alpha$-stable random walk.
\end{abstract}

\section{Introduction}
\label{sec:introduction}

Let $\calL$ be the law of a random point process on $\R$. Brunet, Derrida et al. introduced in \cite{BrD97,BDMM07} a discrete-time branching-selection particle system on $\R$ in which the size of the population is limited by some integer $N$. This process evolves as follows: for any $n \in \N$, every individual alive at the $n^\text{th}$ generation dies giving birth to children around its current position, according to an independent version of a point process of law $\calL$. Only the $N$ children with the largest position are kept alive and form the $(n+1)^\text{st}$ generation of the process. We write $(x^N_n(1), \ldots, x^N_n(N))$ for the positions at time $n$ of individuals in the process, ranked in the decreasing order. This process is called the \textit{$N$-branching random walk}, or $N$-BRW for short.

In \cite{BeG10}, Bérard and Gouéré proved that under some appropriate integrability conditions, the cloud of particles drifts at some deterministic speed
\begin{equation}
  \label{eqn:vNdef}
  v_N := \lim_{n \to +\infty} \frac{x^{N}_n(1)}{n} = \lim_{n \to +\infty} \frac{x^{N}_n(N)}{n} \quad \mathrm{a.s.},
\end{equation}
and obtained the following asymptotic behaviour for $v_N$
\begin{equation}
  \label{eqn:stateoftheart}
  v_\infty - v_N \underset{N \to +\infty}{\sim} \frac{C}{(\log N)^2},
\end{equation}
in which $C$ is an explicit positive constant that depends only on the law $\calL$. Their argument is based on a coupling (recalled in Section~\ref{subsec:coupling}) between the $N$-branching random walk and a branching random walk, that we define now.

A \textit{branching random walk} with branching law $\calL$ is a process defined as follows. It starts with a unique individual located at position 0 at time 0. At each time $k \in \N$, every individual alive in the process at time $k$ dies giving birth to children. The children are positioned around their parent according to i.i.d. point processes with law $\calL$.

We write $\T$ for the genealogical tree of the process. For $u \in \T$, we denote by $V(u)$ the position of $u$, by $|u|$ the time at which $u$ is alive, by $\pi u$ the parent of $u$ (provided that $u$ is not the root of~$\T$) and by $u_k$ the ancestor alive at time $k$ of $u$. We set $\Omega(u)$ the set of siblings of $u$ i.e. the set of individuals $v \in \T$ such that $\pi v = \pi u$ and $v \neq u$. We observe that $\T$ is a (random) Galton-Watson tree with reproduction law $\# \calL$. 

We list assumptions made on the point process law $\calL$. Let $L$ be a point process with law $\calL$. We first assume that the Galton-Watson tree $\T$ is supercritical and a.s. infinite, i.e.
\begin{equation}
  \label{eqn:reproduction}
  \E\left[ \# L \right]>1 \quad \mathrm{and} \quad \P\left(\# L = 0\right) = 0.
\end{equation}
Note that if $\P\left(\# L = 0\right) > 0$, while $\T$ might be infinite with positive probability, the $N$-BRW dies out almost surely. We also suppose the point process law $\calL$ to be in the stable boundary case in the following sense:
\begin{equation}
  \label{eqn:critical1}
  \E\left[ \sum_{\ell \in L} e^{\ell}\right]=1,
\end{equation}
and the random variable $X$ defined by 
\begin{equation}
  \label{eqn:spinal_law}
  \P(X \leq x) = \E\left[ \sum_{\ell \in L} \ind{\ell \leq x} e^\ell \right]
\end{equation}
is in the domain of attraction of a stable random variable $Y$ verifying $\P(Y \geq 0) \in (0,1)$.

Using \cite[Chapter XVII]{Fel}, we provide a necessary and sufficient condition for $X$ to be in the domain of attraction of $Y$. Let $\alpha \in (0,2]$ be such that $Y$ is an $\alpha$-stable random variable verifying $\P(Y \geq 0) \in (0,1)$. We introduce the function 
\begin{equation}
  \label{eqn:defineLstar}
  L^* : x \mapsto x^{\alpha - 2} \E\left[ Y^2 \ind{|Y| \leq x} \right].
\end{equation}
This function is slowly varying\footnote{i.e. for all $\lambda >0$, $\lim_{t \to +\infty} \frac{L^*(\lambda t)}{L^*(t)}=1$.}. We set
\begin{equation}
  b_n = \inf\left\{ x  > 0 : \frac{x^\alpha}{L^*(x)} = n \right\}.
\end{equation}
The random variable $X$ is in the domain of attraction $Y$ if and only if writing $(S_n)$ for a random walk with step distribution with the same law as $X$, $\frac{S_n}{b_n}$ converges in law to $Y$.

Note that if $\E(|X|) < +\infty$, by strong law of large numbers $\frac{S_n}{n} \to \E(X)$ a.s. Thus \eqref{eqn:spinal_law} implies that $\E(X)=0$. In that case, $\calL$ is in the boundary case as defined in \cite{BiK05}. Up to an affine transformation several point process laws verify these properties, adapting the discussion in \cite[Appendix A]{Jaf12} to this setting.

As $Y$ is an $\alpha$-stable random variable, there exists an $\alpha$-stable Lévy process $(Y_t, t \geq 0)$ such that $Y_1$ has the same law as $Y$. Using \cite[Lemma 1]{Mog74}, we define
\begin{equation}
  \label{eqn:defcstar}
  C_* := \lim_{t \to +\infty} -\frac{1}{t} \log \P\left(|Y_s| \leq \frac{1}{2}, s \leq t\right) \in (0,+\infty).
\end{equation}

The next integrability assumption on $\calL$ ensure that the spine of the branching random walk (see Section 2.1) behaves as a typical individual staying close to the boundary of the process:
\begin{equation}
  \label{eqn:integrability}
  \lim_{x \to +\infty} \frac{x^\alpha}{L^*(x)} \E\left[  \sum_{\ell \in L} e^\ell \ind{\log \left(\sum_{\ell' \in L} e^{\ell'} \right) > x + \ell} \right] = 0.
\end{equation}
Finally, we assume that
\begin{equation}
  \label{eqn:lowerIntegrability}
  \E\left[ \left| \max_{\ell \in L} \ell \right|^2 \right] < +\infty,
\end{equation}
this condition is not expected to be optimal but is sufficient to bound from below in a crude way the minimal position in the $N$-BRW, that we use when the coupling fails.

\begin{theorem}
\label{thm:main}
Under the previous assumptions, for an $N$-BRW with reproduction law $\calL$, the sequence $(v_N, N \geq 1)$ defined in \eqref{eqn:vNdef} exists and verifies
\begin{equation}
  \label{eqn:main}
  v_N \underset{N \to +\infty}{\sim}  -C_* \frac{L^*(\log N)}{(\log N)^{\alpha}}.
\end{equation}
\end{theorem}

We observe that if $\calL$ satisfies
\begin{equation}
  \label{eqn:integrability2}
  \E\left[ \sum_{\ell \in L} e^\ell \log \left(\sum_{\ell' \in L} e^{\ell'-\ell} \right)^2\right] + \E\left[ \sum_{\ell \in L} \ell^{2} e^\ell \right] < +\infty,
\end{equation}
then Theorem \ref{thm:main} implies that \eqref{eqn:stateoftheart} holds with $C = \frac{\pi^2}{2} \E\left[ \sum_{\ell \in L} \ell^2 e^\ell \right]$,
which is consistent with the result of Bérard and Gouéré

\paragraph*{Examples.}
We present two point process laws that satisfy the hypotheses of Theorem \ref{thm:main}. Let $X$ be the law of a random variable on $(0,+\infty)$. We write $\Lambda(\theta)$ for the log-Laplace transform of $X$. We assume there exists $\theta^*>0$ such that $\Lambda(\theta^*)= \log 2$, and $\alpha > 1$ verifying
\[
  \P(X \geq x) \sim e^{-\theta^* x} x^{-\alpha - 1}.
\]
In this case, there exists $\mu := \E\left( X e^{\theta^*X} \right)/2$ such that the point process $\calL$ defined as the law of a pair of independent random variables $(Y_1,Y_2)$ which have the same law as $\theta^*(X-\mu)$ satisfies the hypotheses of Theorem \ref{thm:main}.

Let $\nu_\alpha$ be the law of an $\alpha$-stable random variable $Y$ such that $\P(Y \geq 0) \in (0,1)$. If $\tilde{\calL}$ is the law of a Poisson point process on $\R$ with intensity $\nu(dx)e^{-x}$, then $\tilde{\calL}$ satisfies assumptions of Theorem~\ref{thm:main}, and the spine of such a branching random walk is in the domain of attraction of $Y$.

The rest of the article is organised as follows. In Section \ref{sec:lemmas}, we introduce the spinal decomposition, that links the computation of additive branching random walk moments with random walks estimates; and the Mogul'ski\u\i{} small deviations estimate for random walks. In Section \ref{sec:barrier}, these results are used to compute the asymptotic behaviour of the survival probability of a branching random walk with a killing line of slope $-\epsilon$, using the same technique as \cite{GHS11}. This asymptotic is then used in Section \ref{sec:conclusion} to prove Theorem \ref{thm:main}, applying the methods introduced in \cite{BeG10}.

\section{Spinal decomposition and small deviations estimate}
\label{sec:lemmas}

\subsection{The spinal decomposition}

The spinal decomposition is a tool introduced by Lyons, Pemantle and Peres in \cite{LPP95} to study branching processes. It has been extended to branching random walks by Lyons in \cite{Lyo97}. It provides two descriptions of a law absolutely continuous with respect to the law $\P_a$ , of the branching random walk $(\T,V+a)$. We set $W_n = \sum_{|u| = n} e^{V(u)}$ and $\calF_n = \sigma(u,V(u), |u| \leq n)$ the natural filtration on the set of marked trees. By \eqref{eqn:critical1}, $(W_n)$ is a non-negative martingale. We define the probability measure $\bar{\P}_a$ on $\calF_\infty$ such that for any $n \in \N$,
\begin{equation}
  \label{eqn:measurechange}
  \left.\frac{d\bar{\P}_a}{d\P_a}\right|_{\calF_n} = e^{-a} W_n.
\end{equation}
We write $\bar{\E}_a$ for the corresponding expectation.

We construct a second probability measure $\hat{\P}_a$ on the set of marked trees with spine. For $(\T,V)$ a marked tree, we say that $w = (w_n, n \geq 0)$ is a spine of $\T$ if for any $n \in \N$, $|w_n|=n$, $w_n \in \T$ and $(w_n)_{n-1} = w_{n-1}$. We introduce
\begin{equation}
  \label{eqn:measurespine}
  \frac{d\hat{\calL}}{d\calL} = \sum_{\ell \in L} e^{\ell},
\end{equation}
another point process law. The probability measure $\hat{\P}_a$ is the law of the process $(\T,V,w)$ constructed as follows. It starts at time $0$ with a unique individual $w_0$ located at position $a$. It makes children according to a point process of law $\hat{\calL}$. Individual $w_1$ is chosen at random among children $u$ of $w_0$ with probability $e^{V(u)}/W_1$. Similarly, at each generation $n \in \N$, every individual $u$ in the $n^\text{th}$ generation dies, giving birth to children according to independent point processes, with law $\hat{\calL}$ if $u = w_n$ or law $\calL$ otherwise. Finally $w_{n+1}$ is chosen among children $v$ of $w_n$ with probability proportional to $e^{V(v)}$. To shorten notations, we write $\bar{\P}=\bar{\P}_0$, $\hat{\P}=\hat{\P}_0$.
\begin{proposition}[Spinal decomposition]
\label{pro:spinaldecomposition}
Under assumption \eqref{eqn:critical1}, for any $n \in \N$, we have
\[ \left.\hat{\P}_a\right|_{\calF_n} = \left. \bar{\P}_a\right|_{\calF_n}. \]
Moreover, for any $z \in \T$ such that $|z|=n$,
\[ \hat{\P}_a\left( \left. w_n = z \right| \calF_n\right) = \frac{e^{V(z)}}{W_n}, \]
and $(V(w_n),n \geq 0)$ is a random walk starting from $a$, with step distribution defined in \eqref{eqn:spinal_law}.
\end{proposition}

A straightforward consequence of this proposition is the many-to-one lemma. Introduced by Peyrière in \cite{Pey74}, this lemma links additive moments of the branching random walks with random walk estimates. Given $(X_n)$ an i.i.d. sequence of random variables with law defined by \eqref{eqn:spinal_law}, we set $S_n = S_0 + \sum_{j=1}^n X_j$ such that $\P_a(S_0=a)=1$.

\begin{lemma}[Many-to-one lemma]
\label{lem:manytoone}
Under assumption \eqref{eqn:critical1}, for any $n \geq 1$ and measurable non-negative function $g$, we have
\begin{equation}
  \label{eqn:manytoone}
  \E_a\left[ \sum_{|u|=n} g(V(u_1),\cdots , V(u_n))\right] = \E_a\left[ e^{a-S_n} g(S_1,\cdots , S_n) \right].
\end{equation}
\end{lemma}

\begin{proof}
We use Proposition \ref{pro:spinaldecomposition} to compute
\begin{align*}
  \E_a\left[ \sum_{|u|=n} g(V(u_1),\cdots , V(u_n))\right] & = \bar{\E}_a \left[ \frac{e^{a}}{W_{n}} \sum_{|u|=n} g(V(u_1), \cdots , V(u_n)) \right]\\
  & = \hat{\E}_a\left[ e^{a}\sum_{|u|=n} \hat{\P}_a\left( \left. w_n = u \right| \calF_n \right) e^{-V(u)} g(V(u_1),\cdots , V(u_n)) \right]\\
  & = \hat{\E}_a \left[ e^{a-V(w_n)} g(V(w_1), \cdots , V(w_n)) \right].
\end{align*}
We now observe that $(S_n, n \geq 0)$ under $\P_a$ has the same law as $(V(w_n), n \geq 0)$ under $\hat{\P}_a$, which ends the proof.
\end{proof}

The many-to-one lemma can be used to bound the maximal displacement in a branching random walk. For example, for all $y \geq 0$, we have
\begin{align*}
  \E\left[ \sum_{u \in  \T} \ind{V(u) \geq y} \ind{V(u_j) < y, j < |u|} \right]
  &= \sum_{k=1}^{+\infty} \E\left[ \sum_{|u|=k} \ind{V(u) \geq y} \ind{V(u_j) < y, j < |u|} \right]\\
  &= \sum_{k=1}^{+\infty} \E\left[ e^{-S_k} \ind{S_k \geq y} \ind{S_j < y, j < k} \right]\\
  &\leq e^{-y} \sum_{k=1}^{+\infty} \P\left( S_k \geq y, S_j < y, j < k \right) \leq e^{-y}.
\end{align*}
Obviously, this computation leads to
\begin{equation}
  \label{eqn:maxdis}
  \sup_{n \in \N} \P(\max_{|u|=n} V(u) \geq y) \leq \P\left(\max_{u \in \T} V(u) \geq y \right) \leq e^{-y}.
\end{equation}

Using the spinal decomposition, to compute the number of individuals in a branching random walk who stay in a well-chosen path, it is enough to know the probability for a random walk decorated by additional random variables to follow that path.

\subsection{Small deviations estimate and variations}

Let $S$ be a random walk in the domain of attraction of an $\alpha$-stable random variable $Y$. We recall that
\[ L^*(u) = u^{\alpha - 2} \E(Y^2 \ind{|Y| \leq u}) \quad \mathrm{and} \quad \frac{b_n^\alpha}{L^*(b_n)}=n.\]
For any $z \in \R$, we define $\P_z$ such that $S$ under law $\P_z$ has the same law as $S+z$ under law~$\P$. The Mogul'ski\u\i{} small deviation estimate enables to compute the probability for $S$ to present fluctuations of order $o(b_n)$.
\begin{theorem}[Mogul'ski\u\i{} \cite{Mog74}]
\label{thm:mogulskii}
Let $(a_n) \in \R_+^\N$ be such that
\[ \lim_{n \to +\infty} a_n = +\infty, \lim_{n \to +\infty} \frac{a_n}{b_n} = 0. \]
Let $f<g$ be two continuous functions such that $f(0) < 0 < g(0)$. If $\P(Y \leq 0) \in (0,1)$ then
\[
  \lim_{n \to +\infty} \frac{a_n^\alpha}{n L^*(a_n)} \log \P\left[ \frac{S_j}{a_n} \in \left[f\left(j/n\right), g\left(j/n\right)\right], 0 \leq j \leq n \right]
  = - C_* \int_0^1 \frac{ds}{(g(s)-f(s))^\alpha},
\]
where $C_*$ is defined in \eqref{eqn:defcstar}.
\end{theorem}

This result can be seen as a consequence of an $\alpha$-stable version of the Donsker theorem, obtained by Prokhorov. This result yields the convergence of the normalized trajectory of the random walk $S$ to the trajectory of an $\alpha$-stable Lévy process $(Y_t, t \in [0,1])$ such that $Y_1$ has the same law as $Y$.

\begin{theorem}[Prokhorov \cite{Prokhorov}]
\label{thm:donsker}
If $\frac{S_n}{b_n}$ converges in law to a stable random variable $Y$, then
$(\frac{S_\floor{nt}}{b_n}, t \in [0,1])$ converges in law to $(Y_t, t \in [0,1])$ in $\mathcal{D}([0,1])$ equipped with the Skorokhod topology.
\end{theorem}

We observe that the Mogul'ski\u\i{} estimate holds uniformly with respect to the starting point.

\begin{corollary}
\label{cor:mogulskiisup}
With  the same notation as Theorem \ref{thm:mogulskii}, we have
\[
  \lim_{n \to +\infty} \frac{a_n^\alpha}{n L^*(a_n)} \log \sup_{y \in \R} \P_y\left[ \frac{S_j}{a_n} \in [f(j/n), g(j/n)], 0 \leq j \leq n \right]
  = -C_* \int_0^1 \frac{ds}{(g(s)-f(s))^\alpha}.
\]
\end{corollary}

\begin{proof}
Observe in a first time that if $y \not \in [a_n f(0), a_n g(0)]$, then
\[ \P_y\left[ \frac{S_j}{a_n} \in [f(j/n), g(j/n)], 0 \leq j \leq n \right] =0. \]
We now choose $\delta > 0$, and write $K = \ceil{\frac{g(0)-f(0)}{\delta}}$, we have
\[
  \sup_{y \in \R} \P_y\left[ \frac{S_j}{a_n} \in [f(j/n), g(j/n)], 0 \leq j \leq n \right] \leq \max_{k \leq K} \Pi_{f(0)+k\delta, f(0)+(k+1)\delta}(f,g),
\]
where
\begin{align*}
  \Pi_{x, x'}(f,g) &= \sup_{y \in [x a_n, x' a_n]} \P_y\left[ \frac{S_j}{a_n} \in [f(j/n), g(j/n)], 0 \leq j \leq n \right]\\
  &\leq \P\left[ \frac{S_j}{a_n} \in [f(j/n)-x', g(j/n)-x], 0 \leq j \leq n \right].
\end{align*}

Therefore, for all $k \leq K$, we have
\[
  \limsup_{n \to +\infty} \frac{a_n^\alpha}{n L^*(a_n)} \log \Pi_{f(0)+k\delta, f(0)+(k+1)\delta}(f,g) \leq - C_* \int_0^1 \frac{ds}{(g(s)-f(s) + \delta)^\alpha},
\]
which leads to
\[
  \limsup_{n \to +\infty} \frac{a_n^\alpha}{n L^*(a_n)} \log \sup_{y \in \R} \P\left[ \frac{S_j+y}{a_n} \in [f(j/n), g(j/n)], 0 \leq j \leq n \right]\\
  \leq  -C_* \int_0^1 \frac{ds}{(g(s)-f(s)+\delta)^\alpha}.
\]
Letting $\delta \to 0$ concludes the proof, as the lower bound is a direct consequence of Theorem \ref{thm:mogulskii}.
\end{proof}

Using an adjustment of the original proof of Mogul'ski\u\i{}, one can prove a similar estimate for enriched random walks. We set $(X_n,\xi_n)$ a sequence of i.i.d. random variables on $\R\times \R_+$, with $X_1$ in the domain of attraction of the stable random variable $Y$, such that $\P(Y\geq 0) \in (0,1)$. We denote by $S_n = S_0 + X_1 + \cdots + X_n$, which is a random walk in the domain of attraction of $Y$. The following estimate then holds.

\begin{lemma}
\label{lem:mogulskiiinf}
Let $(a_n) \in \R_+^\N$ be such that $\lim_{n \to +\infty} \frac{a_n}{b_n} = 0$. We set $E_n = \{ \xi_j \leq n, j \leq n\}$ and we assume that
\begin{equation}
  \label{eqn:spinelikeassumption}
  \lim_{n \to +\infty} \frac{a_n^\alpha}{L^*(a_n)} \P(\xi_1 \geq n) = 0.
\end{equation}
There exists $C_*>0$, given by \eqref{eqn:defcstar}, such that for any pair $(f,g)$ of continuous functions verifying $f<g$, for any $f(0)<x<y<g(0)$ we have
\[
  \lim_{n \to +\infty} \frac{a_n^\alpha}{n L^*(a_n)} \log \inf_{z \in [x a_n, y a_n]} \P_z\left( \frac{S_j}{a_n} \in \left[ f(j/n) ,g(j/n) \right], j \leq n, E_n\right)\\
  = -C_* \int_0^1 \frac{ds}{(g(s)-f(s))^\alpha}.
\]
\end{lemma}

\begin{proof}
We assume in a first time that $f,g$ are two constant functions. Let $n \geq 1$, $f<x<y<g$ and $f<x'<y'<g$, we denote by
\begin{equation}
  \label{eqn:constants}
  P^{x',y'}_{x,y} (f,g) = \inf_{z \in [x,y]} \P_{za_n} \left( \frac{S_n}{a_n} \in [x',y'], \frac{S_j}{a_n} \in [f,g], j \leq n, E_n  \right).
\end{equation}
Let $A>0$ and $r_n = \floor{A \frac{a_n^\alpha}{L^*(a_n)}}$. We divide $[0,n]$ into $K = \floor{\frac{n}{r_n}}$ intervals of length $r_n$. For any $k \leq K$, we set $m_k = kr_n$ and $m_{K+1}=n$. Applying the Markov property at time $m_K, \ldots, m_1$, and restricting to trajectories which are, at any time $m_k$ in $[x' a_n, y' a_n]$, we have
\begin{equation}
  \label{eqn:betainf}
  P^{x',y'}_{x,y} (f,g) \geq \pi^{x',y'}_{x,y}(f,g) \left(\pi^{x',y'}_{x',y'}(f,g)\right)^{K},
\end{equation}
where we set $\pi^{x',y'}_{x,y}(f,g) = \inf_{z \in [x,y]} \P_{za_n} \left( \frac{S_{r_n}}{a_n} \in [x',y'], \frac{S_j}{a_n} \in [f,g], j \leq r_n, E_{r_n} \right)$.

Let $\delta > 0$ be chosen small enough such that $M = \ceil{\frac{y-x}{\delta}} \geq 3$. We observe easily that
\begin{align}
  \pi^{x',y'}_{x,y}(f,g) &\geq \min_{0 \leq m \leq M} \pi_{x+m\delta,x+(m+1)\delta}^{x',y'}(f, g)\nonumber\\
  &\geq \min_{0 \leq m \leq M} \pi^{x'-(m-1)\delta, y'-(m+1)\delta}_{x,x}(f-(m-1)\delta, g-(m+1)\delta).\label{eqn:alphainf}
\end{align}
Moreover, we have
\begin{align*}
  \pi^{x',y'}_{x,x}(f,g) &= \P_{xa_n}\left( \frac{S_{r_n}}{a_n} \in [x',y'], \frac{S_j}{a_n} \in [f,g], j \leq r_n, E_{r_n} \right)\\
  &\geq \P_{xa_n}\left( \frac{S_{r_n}}{a_n} \in [x',y'], \frac{S_j}{a_n} \in [f,g], j \leq r_n \right) - r_n\P(\xi_1 \geq n).
\end{align*}
By \eqref{eqn:spinelikeassumption}, $\lim_{n \to +\infty} r_n \P(\xi_1 \geq n) = 0$. Moreover, $r_n \sim A \frac{a_n^\alpha}{L^*(a_n)}$ and $X_1$ is in the domain of attraction of $Y$. Thus $\frac{S_{r_n}}{a_n}$ converges in law toward $A^\frac{1}{\alpha}Y$ as $n \to +\infty$. We apply Theorem \ref{thm:donsker}, the process $\left(\frac{S_\floor{r_n t/A}}{a_n}, t \in [0,A]\right)$ converges as $n \to +\infty$ under law $\P_{xa_n}$ to a stable Lévy process $(x+Y_t, t \in [0,A])$ such that $Y_A$ has the same law as $A^{1/\alpha}Y$. In particular
\[
  \liminf_{n \to +\infty} \pi^{x',y'}_{x,x}(f,g) \geq \P_x(Y_{A} \in (x',y'), Y_u \in (f,g), u \leq A).
\]

Using \eqref{eqn:alphainf}, we have
\begin{equation*}
  \liminf_{n \to +\infty} \pi^{x',y'}_{x,y}(f,g) \geq \min_{0 \leq m \leq M}\P_{x+m\delta}(Y_{A} \in (x'+\delta,y'-\delta), Y_u \in (f+\delta,g-\delta), u \leq A ).
\end{equation*}
As a consequence, recalling that $K \sim \frac{n L^*(a_n)}{A a_n^\alpha}$, \eqref{eqn:betainf} leads to
\begin{multline}
  \liminf_{n \to +\infty} \frac{a_n^\alpha}{n L^*(a_n)} \log P_{x,y}^{x',y'}(f,g)\\ \geq
  \frac{1}{A}\min_{0 \leq m \leq M} \log \P_{x'+m\delta}(Y_A \in (x'+\delta,y'-\delta), Y_u \in (f+\delta,g-\delta), u \leq A). \label{eqn:gammainf}
\end{multline}
By \cite[Lemma 1]{Mog74}, we have
\[
  \lim_{t \to +\infty} \frac{1}{t} \log \P_x(Y_t \in (x',y'), Y_s \in (f,g), s \leq t) = -\frac{C_*}{(g-f)^\alpha},
\]
where $C_*$ is defined by \eqref{eqn:defcstar}. Letting $A \to +\infty$ then $\delta \to 0$, \eqref{eqn:gammainf} yields
\begin{equation}
  \liminf_{n \to +\infty} \frac{a_n^\alpha}{n L^*(a_n)} \log P_{x,y}^{x',y'}(f,g) \geq - \frac{C_*}{(g-f)^\alpha}
  \label{eqn:firstestimate}
\end{equation}
which is the expected result when $f,g$ are two constants.

In a second time, we consider two continuous functions $f<g$. Let $f(0)<x<y<g(0)$. We set $h$ a continuous function such that $f<h<g$ and $h(0) = \frac{x+y}{2}$. Let $\epsilon>0$ such that $6 \epsilon \leq \inf_{t \in [0,1]} \min(g(t)-h(t),h(t)-f(t))$. We choose $A>0$ such that
\[ \sup_{|t-s| \leq \frac{2}{A}} |f(t)-f(s)| + |g(t)- g(s)|  + |h(t)-h(s)| \leq \epsilon. \]
and for $a \leq A$, we write $m_a = \floor{an/A}$ and $I_{a,A} = [f(a/A)+\epsilon,g(a/A)-\epsilon]$. We define $J_{0,A} = [x,y]$, and for $1 \leq a \leq A$, $J_{a,A} = [h(a/A)-\epsilon, h(a/A)+\epsilon]$. Applying the Markov property at times $m_{A-1}, \ldots, m_1$, we have
\begin{multline*}
  \inf_{z \in [x a_n, ya_n]} \P_z\left(\frac{S_j}{a_n} \in \left[ f(j/n) ,g(j/n) \right], j \leq n, E_n\right)\\
  \geq \prod_{a=0}^{A-1} \inf_{z \in J_{a,A}} \P_{za_n}\left(
    \frac{S_{m_{a+1}}}{a_n} \in J_{a+1,A}, \frac{S_j}{a_n} \in I_{a,A},j \leq m_{a+1}-m_a, E_{m_{a+1}-m_a} \right).
\end{multline*}
Therefore, using equation \eqref{eqn:firstestimate}, we have
\begin{multline*}
  \liminf_{n \to +\infty} \frac{a_n^\alpha}{nL^*(a_n)} \log \inf_{z \in [xa_n,ya_n]} \P_z\left( \frac{S_j}{a_n} \in \left[ f(j/n) ,g(j/n) \right] , j \leq n, E_n\right)\\
  \geq -\frac{1}{A}\sum_{a=0}^{A-1} C_*\frac{1}{(g(a/A)-f(a/A)-2\epsilon)^\alpha}.
\end{multline*}
As the upper bound is a direct consequence of Theorem \ref{thm:mogulskii}, we let $A \to +\infty$ and $\epsilon \to 0$ to conclude the proof.
\end{proof}

\section{Branching random walk with a barrier}
\label{sec:barrier}

Let $(\T,V)$ be a branching random walk with reproduction law $\calL$ satisfying the hypotheses of Theorem~\ref{thm:main}. We study in this section the asymptotic behaviour, as $n \to +\infty$ and $\epsilon \to 0$ of the quantity
\begin{equation}
  \label{eqn:defineRho}
  \rho(n,\epsilon) = \P\left( \exists |u| = n : \forall j \leq n, V(u_j) \geq - \epsilon j \right).
\end{equation}
The asymptotic behaviour of $\rho(\infty, \epsilon)$ has been studied by Gantert, Hu and Shi in \cite{GHS11} for a branching random walk with a spine in the domain of attraction of a Gaussian random variable. They studied the asymptotic behaviour of $\rho(n,\epsilon)$ for $\epsilon \approx \theta n^{-2/3}$. Using the same arguments, we obtain sharp estimates on the asymptotic behaviour of $\rho(n,\epsilon)$ for $\epsilon \approx \theta \Lambda(n) n^{-\frac{\alpha}{\alpha + 1}}$, where $\Lambda$ is a well-chosen slowly varying function.

We apply the spinal decomposition and the Mogul'ski\u\i{} estimate to compute the number of individuals that stay at any time $k \leq n$ between curves $a_n f(k/n)$ and $a_n g(k/n)$, for an appropriate choice of $(a_n)$, $f$ and $g$. We note that
\begin{align*}
  \E\left[ \sum_{|u|=n} \ind{V(u_j) \in [a_n f(j/n),a_n g(j/n)], j \leq n} \right]
  &= \E\left[ e^{-S_n} \ind{S_j \in [a_n f(j/n), a_n g(j/n), j \leq n} \right]\\
  &\approx e^{-a_n g(1)} \P\left(S_j \in [a_n f(j/n), a_n g(j/n)], j \leq n \right)\\
  &\approx \exp\left(-a_n g(1) - \frac{n L^*(a_n)}{a_n^\alpha} C_* \int_0^1 \frac{ds}{(g(s) -f(s))^\alpha} \right).
\end{align*}
This informal computation hints that to obtain tight estimates, it is appropriate to choose a sequence $(a_n)$ satisfying $a_n \sim_{n \to +\infty} \frac{n L^*(a_n)}{a_n^\alpha}$, and functions $f$ and $g$ verifying
\begin{equation}
  \label{eqn:equadiff}
  \forall t \in [0,1], g(t) + C_* \int_0^t \frac{ds}{(g(s) - f(s))^\alpha} = g(0).
\end{equation}
However, instead of solving explicitly $g'(t) = - C_* (g(t) + \theta t)^{-\alpha}$ as a function of $(t,\theta)$, we use approximate solutions for \eqref{eqn:equadiff}.

For $n \in \N$, we define
\begin{equation}
  \label{eqn:defineAn}
  a_n = \inf\left\{ x \geq 0 : \frac{x^{\alpha + 1}}{L^*(x)} = n \right\}.
\end{equation}
and we introduce the function
\begin{equation}
  \label{eqn:definePhi}
  \Phi : \begin{array}{rcl}
  (0,+\infty)& \longrightarrow &\R\\
  \lambda &\longmapsto & \frac{C_*}{\lambda^\alpha} - \frac{\lambda}{\alpha + 1}.
  \end{array}
\end{equation}
Note that $\Phi$ is a $\calC^\infty$ strictly decreasing function on $(0,+\infty)$, that admits a well-defined inverse $\Phi^{-1}$. The main result of the section is the following.
\begin{theorem}
\label{thm:asymptotic}
Under the assumptions of Theorem \ref{thm:main}, for any $\theta > 0$ we have
\[
  -\frac{C_*^\frac{1}{\alpha}}{\theta^\frac{1}{\alpha}} \leq \liminf_{n \to +\infty} \frac{1}{a_n} \log \rho\left(n, \theta \frac{a_n}{n} \right) \leq \limsup_{n \to +\infty} \frac{1}{a_n} \log \rho\left(n, \theta \frac{a_n}{n} \right) \leq -\Phi^{-1}(\theta).
\]
\end{theorem}

\begin{remark}
For any $\mu>0$ we have $a_\floor{\mu n} \sim_{n \to +\infty} \mu^{\frac{1}{\alpha + 1}} a_n$, by inversion of regularly varying functions. Consequently, Theorem \ref{thm:asymptotic} implies that for any $\theta > 0$,
\begin{multline}
  \label{eqn:encadrement}
  -1 \leq \liminf_{n \to +\infty} \frac{1}{a_n} \log \rho\left( \floor{(\theta/C_*)^\frac{\alpha + 1}{\alpha} n}, C_* \frac{a_n}{n} \right)\\ \leq \limsup_{n \to +\infty} \frac{1}{a_n} \log \rho \left( \floor{(\theta/C_*)^\frac{\alpha + 1}{\alpha} n}, C_* \frac{a_n}{n} \right) \leq -\frac{\theta^{\frac{1}{\alpha}}\Phi^{-1}(\theta)}{C_*^{\frac{1}{\alpha}}} .
\end{multline}
As $\lim_{\theta \to +\infty} \theta^{\frac{1}{\alpha}} \Phi^{-1}(\theta) = C_*^\frac{1}{\alpha}$, this leads to
\begin{equation}
  \label{eqn:encadrementAsymptotic}
  \lim_{h \to +\infty} \liminf_{n \to +\infty}\frac{1}{a_n} \log \rho\left( \floor{h n}, C_* \frac{a_n}{n} \right) = \lim_{h \to +\infty} \limsup_{n \to + \infty}\frac{1}{a_n} \log \rho\left( \floor{h n}, C_* \frac{a_n}{n} \right) = -1.
\end{equation}
\end{remark}

To prove Theorem \ref{thm:asymptotic}, we prove separately an upper bound in Lemma~\ref{lem:upperbound} and the lower bound in Lemma~\ref{lem:lowerbound}. The upper bound is obtained by computing the number of individuals that stay above the line of slope $-\theta \frac{a_n}{n}$ during $n$ units of time.
\begin{lemma}
\label{lem:upperbound}
Under the assumptions of Theorem \ref{thm:main}, for all $\theta > 0$ we have
\[
  \limsup_{n \to +\infty} \frac{1}{a_n} \log \rho\left(n, \theta \frac{a_n}{n} \right) \leq -\Phi^{-1}(\theta).
\]
\end{lemma}

\begin{proof}
Let $\theta > 0$ and $\lambda > 0$, we set $g : t \mapsto -\theta t + \lambda(1-t)^\frac{1}{\alpha + 1}$. For $j \leq n$, we introduce the intervals
\[
  I^{(n)}_j = \left[ - \theta a_n j/n, a_n g(j/n) \right].
\]
As $I^{(n)}_n = \{g(1)a_n\}$, an individual that stays above the curve of slope $-\theta a_n/n$ crosses at some time $k \leq n$ the line $g(./n)a_n$, therefore
\begin{align*}
  \rho\left(n,\theta \frac{a_n}{n}\right) &= \P\left( \exists |u| = n : \forall j \leq n, V(u_j) \geq -\theta a_n \frac{j}{n} \right)\\
  &\leq \P\left( \exists |u| \leq n : V(u) \geq a_n g(|u|/n), V(u_j) \in I^{(n)}_j, j < |u| \right).
\end{align*}
Thus, setting
\[
  Y_n = \sum_{|u| \leq n} \ind{V(u) \geq a_n g(|u|/n)} \ind{V(u_j) \in I^{(n)}_j, j < |u|},
\]
by the Markov inequality we have $\rho\left(n,\theta \frac{a_n}{n}\right) \leq \E(Y_n)$. Applying Lemma \ref{lem:manytoone}, we have
\begin{align*}
  \E(Y_n) &= \sum_{k=1}^n \E\left[ \sum_{|u|=k} \ind{V(u_j) \in I^{(n)}_j, j < k} \ind{V(u) \geq a_n g(k/n)} \right]\\
  &= \sum_{k=1}^n \E\left[ e^{-S_k} \ind{S_j \in I^{(n)}_j, j < k} \ind{S_k \geq a_n g(k/n)} \right]\\
  &\leq \sum_{k=1}^n e^{-g(k/n)a_n} \P\left( S_j \in I^{(n)}_j, j < k\right).
\end{align*}
Let $A \in \N$, we set $m_a = \floor{na/A}$ and $g_{a,A} = \inf_{s \in [\frac{a-1}{A},\frac{a+2}{A}]} g(s)$, we have
\begin{equation*}
  \E(Y_n) \leq \sum_{a=0}^{A-1} \sum_{k=m_a+1}^{m_{a+1}} e^{-g(k/n)a_n} \P\left( S_j \in I^{(n)}_j, j < k\right)
   \leq n \sum_{a=0}^{A-1} e^{-g_{a,A} a_n} \P\left( S_j \in I^{(n)}_j, j \leq m_a \right).
\end{equation*}
Therefore, by Corollary \ref{cor:mogulskiisup}, we have
\begin{align*}
  \limsup_{n \to +\infty} \frac{1}{a_n} \log \E(Y_n) & \leq \max_{a \leq A-1} \left( - g_{a,A} - C_* \int_0^\frac{a}{A} \frac{ds}{(g(s) + \theta s)^\alpha} \right)\\
  &\leq \max_{a \leq A-1} \left( - g_{a,A} - \frac{C_*(\alpha + 1)}{\lambda^\alpha}\left[ 1 - (1 - a/A)^\frac{1}{\alpha + 1} \right] \right).
\end{align*}
Letting $A \to +\infty$, as $g$ is uniformly continuous, we have
\begin{align*}
  \limsup_{n \to +\infty} \frac{1}{a_n} \log \rho\left(n, \theta \frac{a_n}{n} \right) &\leq \sup_{t \in [0,1]} \left\{ \theta t - \lambda (1-t)^\frac{1}{\alpha + 1} - \frac{C_*(\alpha + 1)}{\lambda^\alpha}\left[ 1 - (1 - t)^\frac{1}{\alpha + 1} \right]\right\}\\
  &\leq - \lambda + \sup_{t \in [0,1]} \left\{ \theta t - (\alpha + 1)\Phi(\lambda)\left[ 1 - (1 - t)^\frac{1}{\alpha + 1} \right] \right\}.
\end{align*}

Note that $t \mapsto  1 - (1 - t)^\frac{1}{\alpha + 1}$ is a convex function with slope $\frac{1}{\alpha + 1}$ at $t=0$. Therefore, if we choose $\lambda = \Phi^{-1}(\theta)$, the function $t \mapsto \theta t - (\alpha + 1)\Phi(\lambda)\left[ 1 - (1 - t)^\frac{1}{\alpha + 1} \right]$ is concave and decreasing. As a consequence
\[
  \limsup_{n \to +\infty} \frac{1}{a_n} \log \rho\left(n, \theta \frac{a_n}{n} \right) \leq - \lambda,
\]
which concludes the proof.
\end{proof}

To obtain a lower bound, we bound from below the probability for an individual to stay between two given curves, while having not too many children. To do so, we compute the first two moments of the number of such individuals, and apply the Cauchy-Schwarz inequality to conclude.

\begin{lemma}
\label{lem:lowerbound}
Under the assumptions of Theorem \ref{thm:asymptotic}, for all $\theta > 0$ we have
\[
  \liminf_{n \to +\infty} \frac{1}{a_n} \log \rho\left(n, \theta \frac{a_n}{n} \right) \geq - \frac{C_*^\frac{1}{\alpha}}{\theta^\frac{1}{\alpha}}.
\]
\end{lemma}

\begin{proof}
For $u \in \T$, we recall that $\Omega(u) = \left\{ v \in \T : \pi v = \pi u \text{ and } v \neq u \right\}$ is the set of siblings of $u$. We introduce $\xi(u) = \log \sum_{v \in \Omega(u)} e^{V(v)-V(u)}$. Note that \eqref{eqn:integrability} implies
\begin{equation}
  \label{eqn:integCons}
  \lim_{x \to +\infty} \frac{x^\alpha}{L^*(x)}\hat{\P}\left( \xi(w_1) \geq x \right) = 0.
\end{equation}

Let $\theta > 0$, $\lambda > 0$ and $\delta > 0$. For $j \leq n$, we set $I^{(n)}_j = \left[ - a_n \theta j/n, a_n(\lambda - \theta j/n) \right]$ and
\[
  X_n = \sum_{|u|=n} \ind{V(u_j) \in I^{(n)}_j, j \leq n} \ind{\xi(u_j) \leq \delta a_n, j \leq n}.
\]
We observe that
\begin{align*}
  \rho\left(n,\theta \frac{a_n}{n}\right)
  &= \P\left( \exists |u| = n : V(u_j) \geq -a_n \theta j/n, j \leq n \right)\\
  &\geq \P\left( \exists |u| = n : V(u_j) \in I^{(n)}_j, j \leq n \right) \geq \P\left( X_n \geq 1\right),
\end{align*}
thus by the Cauchy-Schwarz inequality, $\rho\left( n, \theta \frac{a_n}{n} \right) \geq \frac{\left(\E(X_n)\right)^2}{\E(X_n^2)}$.

In a first time, we bound from below $\E(X_n)$. Using Proposition \ref{pro:spinaldecomposition}, we have
\begin{align*}
  \E(X_n)
  &= \bar{\E}\left[ \frac{1}{W_n} \sum_{|u|=n} \ind{V(u_j) \in I^{(n)}_j , j \leq n} \ind{\xi(u_j) \leq \delta a_n, j \leq n} \right]\\
  &= \hat{\E}\left[ \sum_{|u|=n} e^{-V(u)} \hat{\P}(u = w_n | \calF_n)\ind{V(u_j) \in I^{(n)}_j , j \leq n} \ind{\xi(u_j) \leq \delta a_n, j \leq n} \right]\\
  &= \hat{\E}\left[ e^{-V(w_n)} \ind{V(w_j) \in I^{(n)}_j, j \leq n} \ind{\xi(w_j) \leq \delta a_n, j \leq n} \right].
\end{align*}
Let $\epsilon \in (0,\lambda)$, as $I^{(n)}_n = [-\theta a_n, (\lambda-\theta)a_n]$ we have
\begin{align*}
  \E(X_n) &\geq \hat{\E}\left[ e^{-V(w_n)} \ind{V(w_n) \leq (\epsilon -\theta) a_n} \ind{V(w_j) \in I^{(n)}_j, j \leq n} \ind{\xi(w_j) \leq \delta a_n, j \leq n} \right]\\
  &\geq e^{(\theta - \epsilon) a_n} \hat{\P}\left[V(w_n) \leq (\epsilon - \theta) a_n, V(w_j) \in I^{(n)}_j,\xi(w_j) \leq \delta a_n, j \leq n\right].
\end{align*}
We introduce $0<x<y$ and $A>0$ such that $\hat{\P}(V(w_1) \in [x,y], \xi(w_1) \leq A) > 0$. Applying the Markov property at time $p=\floor{\epsilon a_n}$, for any $n \geq 1$ large enough we have
\begin{multline*}
  \hat{\P}\left[V(w_j) \in I^{(n)}_j,\xi(w_j) \leq \delta a_n, j \leq n\right]\\
  \geq \hat{\P}(V(w_1) \in [x,y], \xi(w_1) \leq A)^p \inf_{z \in [x \epsilon a_n, y \epsilon a_n]}\hat{\P}_z \left[ V(w_j) \in I^{(n)}_{j+p},\xi(w_j) \leq \delta a_n, j \leq n-p \right].
\end{multline*}
As \eqref{eqn:integCons} holds, we apply Lemma \ref{lem:mogulskiiinf},
\[
  \liminf_{n \to +\infty} \frac{1}{a_n} \log \E(X_n) \geq \theta - \epsilon - \frac{C_*}{\lambda^\alpha} + \epsilon \log \hat{\P}(V(w_1) \in [x,y], \xi(w_1) \leq A).
\]
Letting $\epsilon \to 0$, we have
\[
 \liminf_{n \to +\infty} \frac{1}{a_n} \log \E(X_n) \geq \theta - \frac{C_*}{\lambda^\alpha}.
\]

To bound from above the second moment of $X_n$, we apply once again the spinal decomposition,
\begin{align*}
  \E(X_n^2) &= \bar{\E}\left( \frac{X_n}{W_n} \sum_{|u|=n} \ind{V(u_j) \in I^{(n)}_j , j \leq n} \ind{\xi(u_j) \leq \delta a_n, j \leq n } \right)\\
  &= \bar{\E}\left( X_n \sum_{|u|=n} e^{-V(u)} \hat{\P}(w_n = u | \calF_n) \ind{V(u_j) \in I^{(n)}_j , j \leq n} \ind{\xi(u_j) \leq \delta a_n, j \leq n }\right)\\
  &= \hat{\E}\left( e^{-V(w_n)} X_n \ind{V(w_j) \in I^{(n)}_j, j \leq n} \ind{\xi(w_j) \leq \delta a_n, j \leq n} \right)\\
  &\leq e^{\theta a_n} \hat{\E}\left[ X_n \ind{V(w_j) \in I^{(n)}_j, j \leq n} \ind{\xi(w_j) \leq \delta a_n, j \leq n} \right].
\end{align*}
We decompose the set of individuals counted in $X_n$ under law $\hat{\P}$ according to their most recent common ancestor with the spine $w$, we have
\[
  X_n = \ind{V(w_j) \in I^{(n)}_j, j \leq n} \ind{\xi(w_j) \leq \delta a_n, j \leq n} + \sum_{j=1}^{n} \sum_{u \in \Omega(w_j)} \Lambda(u),
\]
where $u' \geq u$ means $u'$ is a descendant of $u$ and
\[
  \Lambda(u) = \sum_{|u'|=n, u' \geq u} \ind{V(u'_j) \in I^{(n)}_j , j \leq n} \ind{\xi(u'_j) \leq \delta a_n, j \leq n}.
\]

We write 
\[\calG= \sigma\left((w_k, \Omega(w_k), V(u), u \in \Omega(w_k)), k \geq 0 \right)\]
for the sigma-field of the information of the spine. Let $k \leq n$ and $u \in \Omega(w_k)$. Conditionally on $\calG$, the subtree rooted at $u$ with marks $V$ is a branching random walk with law $\P_{V(u)}$, therefore
\begin{align*}
  \hat{\E}\left(\left.\Lambda(u) \right| \calG\right)
  &\leq \E_{V(u)}\left( \sum_{|u'|=n-k} \ind{V(u'_j) \in I^{(n)}_{k+j}, j \leq n-k} \right)\\
  &\leq e^{V(u)} \E_{V(u)} \left( e^{-S_{n-k}} \ind{S_j \in I^{(n)}_{k+j}, j \leq n-k} \right]\\
  &\leq e^{V(u)}e^{\theta a_n} \sup_{z \in \R}\P_{z}\left( S_j \in I^{(n)}_{k+j} , j \leq n-k \right).
\end{align*}
Let $A \in \N$, we set $m_a = \floor{na/A}$ and $\Psi_{a,A} = \sup_{z \in \R} \P_z\left( S_j \in I^{(n)}_{m_a+j}, j \leq n-m_a \right)$. For any $k \leq m_{a+1}$ and $u \in \Omega(w_k)$, we have $\hat{\E}\left(\left.\Lambda(u) \right| \calG\right) \leq e^{V(u)}e^{\theta a_n} \Psi_{a+1,A}$, thus
\begin{align*}
  &\hat{\E}\left[  \ind{V(w_j) \in I^{(n)}_j, j \leq n} \ind{\xi(w_j) \leq \delta a_n, j \leq n}\sum_{k=m_a+1}^{m_{a+1}} \sum_{u \in \Omega(w_k)} \Lambda(u) \right]\\
  \leq & \sum_{k=m_a+1}^{m_{a+1}} \hat{\E}\left[ \ind{V(w_j) \in I^{(n)}_j, j \leq n} \sum_{u \in \Omega(w_k)} \ind{\xi(w_k) \leq \delta a_n} \Lambda(u) \right]\\
  \leq & \Psi_{a+1,A} e^{\theta a_n} \sum_{k=m_a+1}^{m_{a+1}} \hat{\E}\left[\ind{V(w_j) \in I^{(n)}_j, j \leq n} e^{\xi(w_k) + V(w_k)} \ind{\xi(w_k) \leq \delta a_n} \right]\\
  \leq & n\Psi_{a+1,A} \Psi_{0,A} e^{(\lambda + (1 - a/A) \theta +\delta) a_n}.
\end{align*}
Consequently, applying Corollary \ref{cor:mogulskiisup}, as soon as $\theta \geq \frac{C_*}{\lambda^\alpha}$ we have
\begin{equation*}
  \limsup_{n \to +\infty} \frac{1}{a_n} \log \E(X_n^2) \leq \max_{a \leq A} \left(\lambda + (2- a/A) (\theta- \frac{C_*}{\lambda^\alpha}) + \delta\right) \leq \lambda + 2 \theta - 2 \frac{C_*}{\lambda^\alpha} + \delta.
\end{equation*}

Using the first and second moment estimates of $X_n$, we have
\[  \liminf_{n \to +\infty} \frac{1}{a_n} \log \rho\left(n, \theta \frac{a_n}{n} \right) \geq -\lambda - \delta.\]
Letting $\delta \to 0$ and $\lambda \to (C_*/\theta)^\frac{1}{\alpha}$ concludes the proof.
\end{proof}

\begin{remark}
If we assume $(f^\theta,g^\theta)$ to be a pair of functions solution of the differential equation
\[
  \begin{cases}
    f(t) = - \theta t\\
    g(t) = -\theta + C_* \int_t^1 \frac{ds}{(g(s)-f(s))^\alpha},
  \end{cases}
\]
using similar estimates as the ones developed in Lemmas \ref{lem:upperbound} and \ref{lem:lowerbound}, we prove that for all $\theta \in \R$
\[
  \lim_{n \to +\infty} \frac{1}{a_n} \log \rho\left(n,\theta \frac{a_n}{n}\right) = -g^\theta(0).
\]
Theorem \ref{thm:asymptotic} is used to obtain bounds for $g^\theta(0)$ admitting a closed expression, that are precise for large $\theta$. Using similar methods, applied to different functions, we also obtain estimates on the behaviour of $g^\theta(0)$ for small values of $\theta$, namely
\[
  \lim_{\theta \to 0} g^\theta(0) = \left( (\alpha + 1)C_* \right)^\frac{1}{\alpha + 1}.
\]
\end{remark}

\section{Speed of the \textit{N}-branching random walk}
\label{sec:conclusion}

In \cite{BeG10}, to prove that $\lim_{n \to +\infty} (\log N)^2 v_N = C$ for a branching random walk in the usual boundary case, the essential tool was a version of Theorem~\ref{thm:asymptotic}, found in \cite{GHS11}. The same methods are applied to compute the asymptotic behaviour of $v_N$ under the assumptions of Theorem~\ref{thm:main}. Loosely speaking, we compare the $N$-BRW with $N$ independent branching random walks in which individuals crossing a linear boundary with slope $-\nu_N$ defined by
\begin{equation}
  \label{eqn:defineNuN}
  \nu_N := C_* \frac{L^*(\log N)}{(\log N)^\alpha}.
\end{equation}
By \eqref{eqn:encadrementAsymptotic}, for any $h>0$ and $N \geq 1$ large enough, $\rho\left(h \frac{(\log N)^{\alpha + 1}}{L^*(\log N)}, \nu_N\right) \approx \frac{1}{N}$. Thus $ \frac{(\log N)^{\alpha + 1}}{L^*(\log N)}$ is expected to be the correct time scale for the study of the process.

We start this section with a more precise definition of the branching-selection particle system we consider. We introduce additional notation that enables to describe it as a measure-valued Markov process. In Section \ref{subsec:coupling}, we introduce an increasing coupling between branching-selection particles systems, and use it to prove the existence of $v_N$. Finally, we obtain in Section \ref{subsec:ub} an upper bound for $v_N$ and in Section \ref{subsec:lb} a lower bound, that are enough to conclude the proof of Theorem \ref{thm:main}.

\subsection{Definition of the \textit{N}-branching random walk and notation}

The branching-selection models we consider are particle systems on $\R$. It is often convenient to represent the state of a particle system by a counting measure on $\R$ with finite integer-valued mass on every interval of the form $[x,+\infty)$. The set of such measures is written $\calM$. A Dirac mass at position $x \in \R$ indicates the presence of an individual alive at position $x$. With this interpretation, a measure in $\calM$ represents a population with a rightmost individual, and no accumulation point. For $N \in \N$, we write $\calM_N$ for the set of measures in $\calM$ with total mass $N$, that represent populations of $N$ individuals. If $\mu \in \calM_N$, then there exists $(x_1,\ldots, x_N) \in \R^N$ such that $\mu = \sum_{j=1}^N \delta_{x_j}$.

We introduce a partial order on $\calM$: given $\mu,\nu \in \calM$, we write $\mu \preccurlyeq \nu$ if for all $x \in \R$, $\mu([x,+\infty)) \leq \nu([x, +\infty))$. Note that if $\mu \preccurlyeq \nu$ then $\mu(\R) \leq \nu(\R)$. A similar partial order can be defined on the set of laws point processes. We say that $\calL \preccurlyeq \tilde{\calL}$ if there exists a coupling $(L,\tilde{L})$ of these two laws, such that $L$ has law $\calL$, $\tilde{L}$ has law $\tilde{\calL}$ and
\[
  \sum_{\ell \in L} \delta_{\ell} \preccurlyeq \sum_{\tilde{\ell} \in \tilde{L}} \delta_{\tilde{\ell}} \quad \mathrm{a.s.}
\]

Let $N \in \N$. We introduce a Markov chain $(X_n^N, n \geq 0)$ on $\calM_N$ called the $N$-BRW. For any $n \geq 0$, we denote by $(x_n^N(1),\ldots, x_n^N(N)) \in \R^N$ the random vector that verifies
\[
  X_n^N = \sum_{j=1}^N \delta_{x_n^N(j)} \quad \mathrm{and} \quad x_n^N(1) \geq x_n^N(2) \geq \cdots \geq x_n^N(N).
\]
Conditionally on $X_n^N$, $X_{n+1}^{N}$ is constructed as follows. Let $(L^1_n,\ldots, L^N_n)$ be $N$ i.i.d. point processes with law $\calL$, we set
\[
  Y_{n+1}^N = \sum_{i=1}^N \sum_{\ell^i \in L^i_n} \delta_{x_n^N(i) + \ell^i} \in \calM,
\]
which is the population after the branching step. We set $y = \sup\{ x \in \R : Y_{n+1}^N ([x,+\infty)) \geq N \}$ and $P = Y_{n+1}^N((y,+\infty))$. We write $X_{n+1}^N = {Y_{n+1}^N}_{|(y,+\infty)} + (N-P)\delta_y$. The natural filtration associated to the $N$-BRW is defined, for $n \in \N$, by $\calF_n = \sigma(L^1_j,\ldots, L^N_j, j \leq n)$. Whereas this is not done here, genealogical informations can freely be added to this process; breaking ties in any $\calF$-adapted manner to choose which of the individuals at the leftmost position are killed.

\subsection{Increasing coupling of branching-selection models}
\label{subsec:coupling}

We construct here a coupling between $N$-BRWs, that preserves the order $\preccurlyeq$. This coupling has been introduced in \cite{BeG10}, in a special case and is a key tool in the study of the branching-selection processes we consider. It is used to bound from above and from below the behaviour of the $N$-BRW by a branching random walk in which individuals that cross a line of slope $-\nu_N$ are killed. In a first time, we couple a single step of the $N$-BRW.

\begin{lemma}
Let $1 \leq m \leq n$ and $\mu \in \calM_m, \tilde{\mu} \in \calM_n$ be such that $\mu \preccurlyeq \tilde{\mu}$. Let $\calL \preccurlyeq \tilde{\calL}$ be two laws of point processes. For any $1 \leq M \leq N$, there exists a coupling of $X^M_1$ the first step of an $M$-BRW with reproduction law $\calL$ starting from $\mu$ with $\tilde{X}^N_1$ the first step of an $N$-BRW with reproduction law $\tilde{\calL}$ starting from $\tilde{\mu}$, in a way that $X^M_1 \preccurlyeq \tilde{X}^N_1$ a.s.
\end{lemma}

\begin{proof}
Let $(L,\tilde{L})$ be a pair of point processes such that $\sum_{\ell \in L} \delta_\ell \preccurlyeq \sum_{\ell \in \tilde{L}} \delta_\ell$ a.s., $L$ has law $\calL$ and $\tilde{L}$ has law $\tilde{\calL}$. We set $((L_j,\tilde{L}_j), j \geq 0)$ i.i.d. random variables with the same law as $(L,\tilde{L})$. We write $\mu = \sum_{i=1}^m \delta_{x_i}$ and $\tilde{\mu} = \sum_{i=1}^n \delta_{y_i}$ in a way that $(x_j, j \leq m)$ and $(y_j, j \leq n)$ are ranked in the decreasing order. We set
\[
  \mu^1 = \sum_{i=1}^m \sum_{\ell_i \in L_i} \delta_{x_i + \ell_i} \quad \mathrm{and} \quad \tilde{\mu}^1 = \sum_{i=1}^n \sum_{\ell_i \in \tilde{L}_i} \delta_{y_i + \ell_i}.
\]
We observe that that $\mu^1 \preccurlyeq \tilde{\mu}^1$ a.s.

We set $X^M_1$ for the $M$ individuals with highest positions in $\mu^1$ and $\tilde{X}^N_1$ the $N$ individuals with highest positions in $\tilde{\mu}^1$. Once again, we have $X^M_1 \preccurlyeq \tilde{X}^N_1$ a.s.
\end{proof}

A direct consequence of this lemma is the existence of an increasing coupling between $N$-BRWs.
\begin{corollary}
\label{cor:increasing}
Let $\calL \preccurlyeq \tilde{\calL}$ be two laws of point processes. For all $1 \leq M \leq N \leq +\infty$, if $X^M_0 \preccurlyeq \tilde{X}^N_0$, then there exists a coupling between the $M$-BRW $(X_n^M)$ with law $\calL$ and the $N$-BRW $(\tilde{X}_n^N)$ with law $\tilde{\calL}$ verifying
\[
  \forall n \in \N, X_n^M \preccurlyeq \tilde{X}_n^N \quad \mathrm{a.s.}
\]
\end{corollary}

Using this increasing coupling, we prove that with high probability, the cloud of particles in the $N$-BRW does not spread.
\begin{lemma}
\label{lem:cloudsize}
Under the assumptions \eqref{eqn:reproduction}, \eqref{eqn:critical1} and \eqref{eqn:lowerIntegrability} there exists $C>0$ such that for all $N \geq 2$, $y \geq 1$ and $n \geq C (\log N+ \log y)$,
\[
  \P\left(x^N_n(1) - x^N_n(N) \geq y \right) \leq C \left(\frac{N (\log N+\log y)}{y}\right)^2.
\]
\end{lemma}

\begin{proof}
Let $n \in \N$ and $k \leq n$, we bound $x^N_n(1) - x^N_{n-k}(1)$ from above and $x^N_n(N)-x^N_{n-k}(1)$ from below to estimate the size of the cloud of particles at time $n$. An appropriate choice of $k$ concludes the proof of Lemma \ref{lem:cloudsize}.

We first observe that the $N$-BRW starting from position $X^N_{n-k}$ can be coupled with $N$ i.i.d. branching random walks $((\T^j, V^j), j \leq N)$ with $(\T^j,V^j)$ starting from position $x^N_{n-k}(j)$, in a way that
\[
  X^N_n \preccurlyeq \sum_{j=1}^N \sum_{u \in \T^j, |u|=k} \delta_{V^j(u)}.
\]
As a consequence, by \eqref{eqn:maxdis}, for any $y \in \R$ and $k \leq n$
\begin{equation}
  \label{eqn:ub}
  \P\left(x^N_n(1) - x^N_{n-k}(1) \geq y\right) \leq \P\left( \max_{j \leq N} \max_{u \in \T^j, |u| = k} V^j(u) \geq y \right) \leq N e^{-y}.
\end{equation}

We now bound from below the displacements in the $N$-BRW. Let $L$ be a point process with law $\calL$. By \eqref{eqn:reproduction}, there exists $R>0$ such that $\E\left( \sum_{\ell \in L} \ind{\ell \geq -R} \right) > 1$. We denote by $L_R$ the point process that consists in the maximal point in $L$ as well as any other point that is greater than $-R$. Using Corollary \ref{cor:increasing}, we couple $(X^N_{n-k+m}, m \geq 0)$ with the $N$-BRW $(X^{N,R}_m, m \geq 0)$ of reproduction law $\calL_R$, starting from a unique individual located at $x^N_{n-k}(1)$ at time $0$ in an increasing fashion.

As $X^{N,R}_m \preccurlyeq X^N_{n-k+m}$, if $X^{N,R}_k(\R) = N$, then $x^{N,R}_k(N) \leq x^N_{n}(N)$. Moreover by definition of $L_R$, the minimal displacement made by one child with respect to its parent is given by $\min(-R, \max L)$. For $n \in \N$, we write $Q_{n}$ a random variable defined as the sum of $n$ i.i.d. copies of $\min(-R, \max L)$. Observe that $Q_{kN}$ is stochastically dominated by $x^{N,R}_{k}(N)-x^N_{n-k}(1)$. Consequently
\[
  \P\left( x^N_n(N)-x^N_{n-k}(1) \leq -y \right) \leq \P\left( X^{N,R}_k(\R) < N \right) + \P\left( Q_{kN} \leq -y \right).
\]
By \eqref{eqn:lowerIntegrability}, we have $\P(Q_{kN} \leq -y) \leq C\frac{k^2 N^2}{y^2}$. Moreover the process $(X^{N,R}_n(\R), n \geq 0)$ is a Galton-Watson process with reproduction law given by $\# L_R$, that saturates at $N$. We set $m_R = \E(\# L_R)$ and $\alpha = - \frac{\log \P(\# L_R = 1)}{\log m_R}$. We have $\P( X^{N,R}_k(\R) < N ) \leq C N^\alpha m_R^{-k \alpha}$, by \cite{FlW07}.
We conclude that
\begin{equation}
  \label{eqn:lb}
  \P\left( x^N_n(N) - x^N_{n-k}(1) \leq y \right) \leq C\frac{k^2 N^2}{y^2} + C \frac{N^\alpha}{m_R^{k \alpha}}.
\end{equation}

Combining \eqref{eqn:ub} and \eqref{eqn:lb}, for all $y \geq 1$ and $k \in \N$ we have
\[
  \P\left( x^N_n(1) - x^N_n(N) \geq 2y \right) \leq N e^{-y} +  C\frac{k^2 N^2}{y^2} + C \frac{N^\alpha}{m_R^{k\alpha}}.
\]
Thus, setting $k = \floor{\frac{3(\log N + \log y)}{\alpha \log m_R}}$, there exists $C>0$ such that for any $y \geq 1$ and $N \geq 1$ large enough, for any $n \geq k$,
\[
  \P\left( x^N_n(1) - x^N_n(N) \geq 2y \right) \leq C \left(\frac{N (\log N+\log y)}{y}\right)^2.
\]
\end{proof}

Applying Lemma \ref{lem:cloudsize} and the Borel-Cantelli lemma, for any $N \geq 2$ we have
\[
  \lim_{n \to +\infty} \frac{x^N_n(1) - x^N_n(N)}{n} = 0 \quad \text{a.s. and in } L^1.
\]

\begin{lemma}
\label{lem:existenceSpeed}
Under the assumptions \eqref{eqn:reproduction}, \eqref{eqn:critical1} and \eqref{eqn:lowerIntegrability}, for any $N \geq 1$, there exists $v_N \in \R$ such that for all $j \leq N$
\begin{equation}
  \label{eqn:existenceSpeed}
  \lim_{n \to +\infty} \frac{x^N_n(j)}{n} = v_N \quad \text{a.s. and in }L^1.
\end{equation}
Moreover, if $X^N_0 = N \delta_0$, we have
\begin{equation}
  \label{eqn:encadrementSpeed}
  v_N = \inf_{n \geq 1} \frac{\E(x_n^N(1))}{n} = \sup_{n \geq 1} \frac{\E(x_n^N(N))}{n}.
\end{equation}
\end{lemma}

\begin{proof}
This proof is based on the Kingman's subadditive ergodic theorem. We first prove that if $X^N_0 = N \delta_0$, then $(x^N_n(1))$ is a subadditive sequence, and $(x^N_n(N))$ is an overadditive one. Thus $\frac{x^N_n(1)}{n}$ and $\frac{x^N_n(N)}{n}$ converge, and $\lim_{n \to +\infty} \frac{x^N_n(1)}{n} = \lim_{n \to +\infty} \frac{x^N_n(N)}{n}$ a.s. by Lemma \ref{lem:cloudsize}. We treat in a second time the case of a generic starting value $X^N_0 \in \calM_N$, using Corollary \ref{cor:increasing}.

Let $N \in \N$, let $(L^j_{n}, j \leq N, n \geq 0)$ be an array of i.i.d. point processes with common law $\calL$. We define on the same probability space random measures $(X^N_{m,n}, 0 \leq m \leq n)$ such that for all $m \geq 0$, $(X^N_{m,m+n}, n \geq 0)$ is an $N$-BRW starting from the initial distribution $N \delta_0$. For any $m \geq 0$, we set $X^N_{m,m} = N \delta_0$. Let $0 \leq m \leq n$, we assume that $X^N_{m,n} = \sum_{j=1}^N \delta_{x^N_{m,n}(j)}$, with $(x^N_{m,n}(j))$ listed in the decreasing order, is given. We define $(x^N_{m,n+1}(j), j \geq 0)$, again listed in the decreasing order, in a way that
\[
  \sum_{j=1}^{+\infty} \delta_{x^N_{m,n+1}(j)} = \sum_{j=1}^N \sum_{\ell^j_{n} \in L^j_{n}} \delta_{x^N_{m,n}(j) + \ell^j_{n}},
\]
and set $X^N_{m,n+1} = \sum_{j=1}^N \delta_{x^N_{m,n+1}(j)}$.

For $x \in \R$, we write $\phi_x$ for the shift operator on $\calM$, such that $\phi_x(\mu) = \mu(.-x)$. With this definition, we observe that for any $0 \leq m \leq n$ we have
\[
  \phi_{x^N_{0,n}(N)}\left(X^N_{n,n+m}\right) \preccurlyeq X^N_{0,n+m} \preccurlyeq  \phi_{x^N_{0,n}(1)}\left(X^N_{n,n+m}\right).
\]
As a consequence,
\begin{equation}
  \label{eqn:subadditive}
  x^N_{0,n+m}(1) \leq x^N_{0,n}(1) + x^N_{n,n+m}(1) \quad \mathrm{and} \quad x^N_{0,n+m}(N) \geq x^N_{0,n}(N) + x^N_{n,n+m}(N).
\end{equation}
We apply Kingman's subadditive ergodic theorem. Indeed for any $n \geq 0$, $(x^N_{n,n+m}(1), m \geq 0)$ is independent of $(x^N_{k,l}(1), 0 \leq k \leq l \leq n)$ and has the same law as $(x^N_{0,m}(1), m \geq 0)$. Moreover, $\E(|x^N_{0,1}(1)|) < +\infty$ by \eqref{eqn:lowerIntegrability}. As a consequence, \eqref{eqn:subadditive} implies there exists $v_N \in \R$ verifying
\[
  \lim_{n \to +\infty} \frac{x^N_{0,n}(1)}{n} = v_N \quad \text{a.s. and in } L^1,
\]
and $v_N = \inf_{n \in \N} \frac{\E(x^N_{0,n}(1))}{n}$. Similarly, $\lim_{n \to +\infty} \frac{x^N_{0,n}(N)}{n} = \sup_{n \in \N} \frac{\E(x^N_{0,n}(N))}{n}$ a.s. and in $L^1$, which proves that \eqref{eqn:encadrementSpeed} is verified. Moreover, by Lemma \ref{lem:cloudsize}, these limits are equal.

We now consider the general case. Let $(X^N_n, n \geq 0)$ be an $N$-BRW. We couple this process with $Y^N$ and $Z^N$ two $N$-BRWs starting from $N \delta_{x^N_0(1)}$ and $N\delta_{x^N_0(N)}$ respectively, such that for all $n \in \N$, $Z^N_n \preccurlyeq X^N_n \preccurlyeq Y^N_n$. We have
\[
  \forall j \leq N, z^N_n(N) \leq x^N_n(N) \leq x^N_n(j) \leq x^N_n(1) \leq y^N_n(1).
\]
Therefore, for any $j \leq N$, we have
\begin{equation*}
  v_N = \liminf_{n \to +\infty} \frac{z^N_n(N)}{n} \leq \liminf_{n \to +\infty} \frac{x^N_n(j)}{n}
  \leq \limsup_{n \to +\infty} \frac{x^N_n(j)}{n} \leq \limsup_{n \to +\infty} \frac{y^N_n(1)}{n} = v_N \quad \text{a.s.}
\end{equation*}
which yields $\lim_{n \to +\infty} x^N_n(j)/n=v_N$ a.s. Similarly, we have
\begin{align*}
  \E\left[ \left| \tfrac{x^N_n(j)}{n} - v_N \right| \right]& \leq \E\left[ \left( \tfrac{y^N_n(1)}{n}-v_N \right) \ind{x^N_n(j) \geq n v_N} \right] + \E\left[ \left( v_N - \tfrac{z^N_n(N)}{n} \right) \ind{x^N_n(j) \leq n v_N} \right]\\
  &\leq \E\left[ \left| \tfrac{y^N_n(1)}{n}-v_N \right| \right] +  \E\left[ \left| \tfrac{z^N_n(N)}{n}-v_N \right| \right].
\end{align*}
We conclude that $x^N_n(j)/n$ also converges to $v_N$ in $L^1$.
\end{proof}

\begin{remark}
Lemma \ref{lem:existenceSpeed} proves the limit in \eqref{eqn:existenceSpeed} does not depend on the starting position of the $N$-BRW. To prove Theorem \ref{thm:main}, we now study the asymptotic behaviour of $v_N$. This can be done considering only $N$-BRW starting from the initial condition $N\delta_0$.
\end{remark}

To study the asymptotic behaviour of $v_N$ as $N \to +\infty$, we couple the $N$-BRW with a branching random walk in which individuals are killed below the line of slope $-\nu_N$. Applying Theorem \ref{thm:asymptotic}, we derive upper and lower bounds for $v_N$.

\subsection{An upper bound on the maximal displacement}
\label{subsec:ub}

To obtain an upper bound on the maximal displacement in the $N$-branching random walk, we link the existence of an individual alive at time $n$ that made a large displacement with the event there exists an individual staying above a line of slope $-\nu_N$ during $m_N$ units of time in a branching random walk. The following lemma is an easier and less precise version of \cite[Lemma 2]{BeG10}, that is sufficient for our proofs.
\begin{lemma}
\label{lem:info}
Let $v < K$. We set $(x_n, n \geq 0)$ a sequence of real numbers with $x_0=0$ such that $\sup_{i \in \N} (x_{i+1}-x_i) \leq K$. For all $m \leq n$, if $x_n > (n-m)v + Km$, then there exists $i \leq n-m$ such that for all $j \leq m$, $x_{i+j}-x_i \geq v j$.
\end{lemma}

\begin{proof}
Let $(x_n)$ be a sequence verifying $\sup_{i \in \N} (x_{i+1}-x_i) \leq K$. We assume that for any $i \leq n-m$, there exists $j_i \leq m$ such that $x_{i+j_i}-x_i \leq vj_i$. We set $\sigma_0 = 0$ and $\sigma_{k+1} = \sigma_k + j_{\sigma_k}$. By definition, we have
\[
  x_{\sigma_{k+1}} \leq (\sigma_{k+1}-\sigma_k) v + x_{\sigma_k},
\]
thus, for all $k \geq 0$, $x_{\sigma_k} \leq \sigma_k v$. Moreover, as $(\sigma_k)$ is strictly increasing, with steps smaller than $m$, there exists $k_0$ such that $\sigma_{k_0} \in [n-m,n]$. We conclude that
\begin{multline*}
  x_n = (x_n - x_{\sigma_{k_0}}) + x_{\sigma_{k_0}} \leq K(n-\sigma_{k_0}) + v \sigma_{k_0} = K n - (K-v)\sigma_{k_0}\\
  \leq Kn - (K-v) (n-m) \leq K m + v(n-m),
\end{multline*}
which concludes the proof.
\end{proof}

Using Lemma \ref{lem:info} and Theorem \ref{thm:asymptotic}, we bound from above the maximal position at time $N^\epsilon$.
\begin{lemma}
\label{lem:ub}
Under the assumptions of Theorem \ref{thm:main}, let $X^N$ be an $N$-BRW with reproduction law $\calL$ starting from $N\delta_0$. For any $\epsilon > 0$ small enough, for any $N \geq 1$ large enough, we have
\[
  \P\left(x^N_\floor{N^{\epsilon}}(1) \geq -(1 - 2\epsilon)\nu_N N^{\epsilon} \right) \leq N^{-\epsilon}.
\]
\end{lemma}

\begin{proof}
Let $\epsilon \in (0,1)$ and $\theta >0$. By \eqref{eqn:encadrement},
\[
  \limsup_{n \to +\infty} \frac{1}{a_n} \log \rho\left( \floor{\left(\frac{\theta}{(1-\epsilon)C_*} \right)^{\frac{\alpha + 1}{\alpha}} n}, C_*(1-\epsilon) \frac{a_n}{n} \right) \leq - \frac{\theta^\frac{1}{\alpha}\Phi^{-1}(\theta)}{C_*^\frac{1}{\alpha}}.
\]
We set $m_N = \floor{\left(\frac{\theta}{(1-\epsilon)C_*} \right)^{\frac{\alpha + 1}{\alpha}} \frac{(\log N)^{\alpha + 1}}{L^*(\log N)}}.$ As $a_{(\log N)^{\alpha + 1}/L^*(\log N)} \sim_{N \to +\infty} \log N$, we have
\[\limsup_{N \to +\infty} \frac{1}{\log N} \log \rho(m_N, (1-\epsilon)\nu_N) \leq - \frac{\theta^\frac{1}{\alpha}\Phi^{-1}(\theta)}{C_*^\frac{1}{\alpha}}.\]
Observe there exists $C>0$ such that $\theta^\frac{1}{\alpha}\Phi^{-1}(\theta) - C_*^\frac{1}{\alpha} \sim_{\theta \to +\infty} - C/\theta$ by definition of $\Phi$. Therefore, for any $\epsilon>0$ small enough,  there exists $\theta>0$ such that $\rho(m_N,(1-\epsilon)\nu_N) \leq N^{-(1 + 2\epsilon)}$ for any $N \geq 1$ large enough.

We set $n = \floor{N^\epsilon}$. Observe the $N$-BRW of length $n$ is built with $nN$ independent point processes of law $\calL$ satisfying \eqref{eqn:critical1}. If $L$ is a point process with law $\calL$, we have
\[
  \P(\max L \geq x) \leq \P\left( \sum_{\ell \in L} e^\ell \geq e^x \right) \leq e^{-x}.
\]
Setting $K = (1+ 2\epsilon) \log N$, the probability there exists one individual in the $N$-BRW alive before time $n$ that made a step larger than $K$ is bounded from above by $1 - (1 - N^{-(1+2\epsilon)})^{nN} \leq N^{-\epsilon}$.

We now consider the path of length $n$ that links an individual alive at time $n$ at position $x^N_n(1)$ with its ancestor alive at time $0$. We write $y^N_n(k)$ for the position of the ancestor at time $k$ of this individual. With probability $1 - N^{-\epsilon}$, this is a path with no step greater than $K$. For $N \geq 1$ large enough, we have $-(1-2\epsilon)\nu_N n > -(n-m_N)(1-\epsilon)\nu_N + Km_N$. Applying Lemma \ref{lem:info}, for any $N \geq 1$ large enough we have
\begin{multline*}
  \left\{ \forall k < n, y^N_n(k+1) - y^N_n(k) \leq K \right\} \cap \left\{x^{N}_n(1) \geq -(1 - 2\epsilon)\nu_N n \right\} \\
  \subset \left\{ \exists j \leq n-m_N : \forall k \leq m_N, y^{N}_n(j+k)-y^{N}_n(j) \geq - (1-\epsilon) \nu_N k \right\}.  
\end{multline*}

Consequently if $x^N_n(1) \geq -(1 - 2\epsilon)\nu_N n$, there exists an individual in the $N$-BRW that has a sequence of descendants of length $m_N$ staying above the line of slope $-(1-\epsilon)\nu_N$. This happens with probability at most $n N \rho(m_N,(1-\epsilon)\nu_N)$. We conclude from these observations that for any $\epsilon>0$ and $N \geq 1$ large enough
\[
  \P\left(x^{N}_n(1) \geq -\nu_N (1- 2\epsilon) n \right) \leq C N^{-\epsilon}.
\]
\end{proof}

\begin{proof}[Proof of the upper bound of Theorem \ref{thm:main}]
Let $X^N$ be an $N$-BRW starting from $N\delta_0$. We note that the maximal displacement at time $n$ in the $N$-BRW is bounded from above by the maximum of $N$ independent branching random walks starting from $0$. By \eqref{eqn:maxdis}, for any $y \geq 0$ and $n \in \N$ we have $\P(x^N_n(1) \geq y) \leq N e^{-y}$.

Moreover, by Lemma \ref{lem:existenceSpeed} we have $\limsup_{n \to +\infty} x^N_n(1)/n \leq \E[x^N_p(1)/p]$ a.s. for all $p \geq 1$. Let $\epsilon>0$ small enough such that Lemma \ref{lem:ub} apply and $y > 0$. Setting $p = \floor{N^{\epsilon/2}}$ we have
\begin{multline*}
  v_N \leq \E\left[\tfrac{x^N_p(1)}{p} \ind{x^N_p(1) \geq p y} \right] + \E\left[\tfrac{x^N_p(1)}{p} \ind{x^N_p(1) \in \left[ -p\nu_N (1- \epsilon),py\right]} \right]\\
  +  \E\left[\tfrac{x^N_p(1)}{p} \ind{x^N_p(1) \leq -p(1-\epsilon)\nu_N} \right],
\end{multline*}
therefore
\begin{align*}
  v_N
  \leq &\int_{y}^{+\infty} \P\left(x^N_p(1) \geq p z \right) dz + y \P\left(x^N_p(1) \geq -p (1-\epsilon) \nu_N \right) - (1 - \epsilon) \nu_N\\
  \leq & N e^{-N^{\epsilon/2}y} + y N^{-\epsilon/2} - (1 - \epsilon) \nu_N.
\end{align*}
Letting $N \to +\infty$ then $\epsilon \to 0$, we conclude that
\[
  \limsup_{N \to +\infty} \frac{v_N (\log N)^\alpha}{L^*(\log N)} \leq -C^*.
\]
\end{proof}

\subsection{The lower bound}
\label{subsec:lb}

To bound from below the position of the leftmost individual in the $N$-BRW, we prove that with high probability, there exists a time $k \leq m_N$ such that $x^N_k(N) \geq -k \nu_N$. We use these events as renewal times for a particle process that stays below the $N$-BRW.

\begin{lemma}
\label{lem:lb}
Under the assumptions of Theorem \ref{thm:main}, let $X^N$ be an $N$-BRW with reproduction law $\calL$ starting from $N\delta_0$. For any $\lambda>0$ and any $\epsilon > 0$ small enough, there exists $\delta>0$ such that for all $N \geq 1$ large enough,
\[
  \P\left( \forall n \leq \lambda \frac{(\log N)^{\alpha + 1}}{L^*(\log N)}, x^N_n(N) \leq - n (1+\epsilon)\nu_N \right) \leq \exp\left( - N^{\delta} \right).
\]
\end{lemma}

\begin{proof}
For $N \in \N$ and $\lambda > 0$, we set $m_N = \floor{\lambda \frac{(\log N)^{\alpha + 1}}{L^*(\log N)}}$. Let $\epsilon > 0$, by \eqref{eqn:encadrement}, we have
\[
  \liminf_{N \to +\infty} \frac{1}{\log N} \log \rho\left(m_N,(1+\epsilon) \nu_N\right) \geq -(1 + \epsilon)^{-\frac{1}{\alpha}}>-1.
\]
Consequently for any $\epsilon>0$, for any $\delta > 0$ small enough we have $\rho(m_N,(1+\epsilon)\nu_N) \geq \frac{1}{N^{1-\delta}}$ for any $N \geq 1$ large enough.

Let $L$ be a point process with law $\calL$. There exists $R>0$ such that $\E\left(\# \left\{ \ell \in L : \ell \geq -R \right\}\right) > 1$. We consider the branching random walk in which individuals that cross the line of slope $-R$ are killed. By standard Galton-Watson processes theory\footnote{See e.g., \cite{FlW07}.}, there exists $r>0$ and $\alpha > 0$ such that for any $N \geq 1$ large enough the probability there exists more than $N$ individuals alive at time $\floor{\alpha \log N}$ in this process is bounded from below by $r$. Thus for all $N \geq 1$ large enough, the probability there exists at least $N+1$ individuals alive at time $m_N + \floor{\alpha \log N}$ in a branching random walk in which individuals that cross the line of slope $-\nu_N(1+2\epsilon)$ are killed is bounded from below by $r \rho(m_N,(1+\epsilon) \nu_N)$.

We set $\calB_N = \left\{ \forall n \leq m_N+\floor{\alpha \log N}, x^N_n(N) \leq - n \nu_N(1 + 2\epsilon) \right\}$. By Corollary \ref{cor:increasing}, the $N$-BRW can be coupled with $N$ independent branching random walks starting from $0$, in which individuals below the line of slope $-\nu_N(1 + 2 \epsilon)$ are killed, in a way that on $\calB_N$, $X^N$ is above the branching random walks for the order $\preccurlyeq$. The probability that at least one of the branching random walks has at least $N+1$ individuals at time $m_N + \floor{\alpha \log N}$ is bounded from below by
\[
 1 - (1-r\rho(m_N,(1+ \epsilon)\nu_N))^{N} \geq 1 - \exp(-N^{\delta/2}),
\]
for any $N \geq 1$ large enough. On this event, the coupling is impossible as $X^N$ has no more that $N$ individuals alive at time $N$, thus $\calB_N$ is not satisfied. We conclude that $\P(\calB_N) \leq e^{-N^{\delta/2}}$.
\end{proof}

\begin{proof}[Proof of the lower bound of Theorem \ref{thm:main}]
The proof is based on a coupling of the $N$-BRW $X^N$ with another particle system $Y^N$, in a way that for any $n \in \N$, $Y^N_n \preccurlyeq X^N_n$. Let $(L^j_{n}, j \leq N, n \geq 0)$ be an array of i.i.d. point processes with law $\calL$. We construct $X^N$ such that $L^j_{n}$ represents the set of children of the individual $x^N_n(j)$, with $X^N_0 = N \delta_0$. By Lemma \ref{lem:lb}, for any $\epsilon>0$ small enough, there exists $\delta > 0$ such that setting $m_N = \floor{\frac{\left( \log N \right)^{\alpha + 1}}{L^*(\log N)}}$, for any $N \geq 1$ large enough we have
\[
  \P\left( \forall n \leq m_N, x^N_n(N) \leq - n (1+\epsilon)\nu_N \right) \leq \exp\left( - N^{\delta} \right).
\]

We introduce $T_0 = 0$ and $Y^N_0 = N \delta_0$. The process $Y^N$ behaves as an $N$-BRW, using the same point processes $(L^j_{n})$ as used for $X^N$ until time
\[
  T_1 = \min\left( m_N, \inf\left\{j \geq 0 : y^N_j(N) > -j\nu_N (1+ \epsilon) \right\}\right).
\]
We then write $Y^N_{T_1^+} = N \delta_{y^N_{T_1}(N)}$, i.e. just after time $T_1$, the process $Y^N$ starts over at time $T_1^+$ from its leftmost individual. For any $k \in \N$, the process behaves as an $N$-BRW between times $T_k^+$ and $T_{k+1}$, defined by
\[
  T_{k+1} = T_k + \min\left( m_N, \inf\left\{j \geq 0 :  y^N_{T_k + j}(N)-y^N_{T_k}(N) > -j\nu_N(1 + \epsilon) \right\} \right).
\]
By construction, for any $k \in \N$ we have $Y^N_k \preccurlyeq X^N_k$ a.s. and in particular $y^N_k(N) \leq x^N_k(N)$.

As $(T_k-T_{k-1}, k \geq 1)$ is a sequence of i.i.d. random variables, Lemma \ref{lem:existenceSpeed} leads to
\[
  \lim_{k \to +\infty} \frac{x^N_{T_k}(N)}{k} = \E(T_1) v_N \quad \mathrm{a.s.}
\]
Moreover, as $(y^N_{T_k}(N) - y^N_{T_{k-1}}(N), k \geq 1)$ is another sequence of i.i.d. random variables, by law of large numbers we have
\[
  \lim_{k \to +\infty} \frac{y^N_{T_k}(N)}{k} = \E(y^N_{T_1}(N)) \quad \mathrm{a.s.}
\]
Combining these two estimates, we have $v_N \geq \frac{\E(y^N_{T_1}(N))}{\E(T_1)}$.

We now compute
\begin{align*}
  \E(y^N_{T_1}(N) ) &= \E\left(y^N_{T_1}(N) \ind{T_1 < m_N} \right) + \E\left( y^N_{T_1}(N) \ind{T_1 = m_N} \right)\\
  &\geq \E\left( -\nu_N(1+\epsilon) T_1 \ind{T_1 < m_N} \right) + \E\left( y^N_{T_1}(N) \ind{T_1 = m_N} \right)\\
  &\geq -\nu_N (1+ \epsilon)\E(T_1) + \E\left( y^N_{T_1}(N) \ind{T_1 = m_N}\right).
\end{align*}
Note that for all $j \leq T_1$, we have $Y^N_j = X^N_j$. Moreover, by Corollary \ref{cor:increasing}, we may couple $X^N$ with an $N$-BRW $\tilde{X}$ in which individuals make only one child, with a displacement of law $\max L$. Consequently, we have
\[
  y^N_{T_1}(N) \geq \left(\sum_{n=1}^{T_1} \min_{j \leq N} \left(\max L_{j,n}\right)\right) \quad \mathrm{a.s.}
\]
which leads to
\[
  v_N \geq -\nu_N (1 + \epsilon) + \frac{1}{\E(T_1)}\E\left[\left(\sum_{n=1}^{m_N} \min_{j \leq N} \left(\max L_{j,n}\right)\right) \ind{T_1 = m_N} \right].
\]
Using the Cauchy-Schwarz inequality and \eqref{eqn:lowerIntegrability}, we have
\[
\E\left[\left(\sum_{n=1}^{m_N} \min_{j \leq N} \left(\max L_{j,n}\right)\right) \ind{T_1 = m_N} \right] \geq - C (N m_N)^{1/2} \P\left(T_1 = m_N\right)^{1/2}.
\]
We apply Lemma \ref{lem:lb} and let $N \to +\infty$ then $\epsilon \to 0$ to prove that
\[
  \liminf_{N \to +\infty} \frac{v_N (\log N)^\alpha}{L^*(\log N)} \geq -C_*.
\]
\end{proof}

\paragraph*{Acknowledgement.} I wish to thank Zhan Shi for introducing me to this topic, and for his help and advices during the research process. I also thank the referee for his help improving the early versions of this manuscript.

\nocite{*}
\bibliographystyle{plain}

\def\cprime{$'$}

\end{document}